\documentclass[12pt]{amsart}
\usepackage[T1]{fontenc}
\usepackage{microtype}

\newcommand{\href}[1]{#1}

\usepackage{fullpage}
\usepackage{ifthen}
\newboolean{PrintVersion}
\setboolean{PrintVersion}{true} 
\usepackage{amsfonts}
\usepackage[mathscr]{eucal}
\usepackage[pdftex]{graphicx}
\usepackage{amsmath,amssymb,amstext,amsthm,mathtools,pgfplots}
\usepackage{mathrsfs}
\usepackage{marginnote} 
\usepackage{color}
\usepackage[all]{xy}

\usepackage{tikz-cd} 
\usepackage{verbatim}

\pgfplotsset{compat=1.12}

\usepackage[pdftex,linktocpage=true,colorlinks,citecolor=blue,linkcolor=blue,pagebackref]{hyperref}
\usepackage{cleveref} 
\usepackage{hyperref}
\usepackage[alphabetic,backrefs]{amsrefs}

\newtheorem{innercustomtheorem}{Theorem}
\newenvironment{customtheorem}[1]
  {\renewcommand\theinnercustomtheorem{#1}\innercustomtheorem}
  {\endinnercustomtheorem}

\newtheorem{thm}{Theorem}[section]

\newtheorem{cor}[thm]{Corollary}

\newtheorem{lemma}[thm]{Lemma}
\newtheorem{prop}[thm]{Proposition}

\theoremstyle{definition}
\newtheorem{defn}[thm]{Definition}
\newtheorem{exam}[thm]{Example}

\newtheorem{notation}[thm]{Notation}
\newtheorem*{ack}{Acknowledgments}

\newtheorem{remark}[thm]{Remark}

\numberwithin{equation}{section}
\DeclareMathOperator{\tr}{Tr}

\DeclareMathOperator{\diag}{diag}

\newcommand{\bb}[1]{\mathbb{#1}}
\newcommand{\bh}{\mathcal{B}(\mathcal{H})}

\newcommand{\cl}[1]{\mathcal{#1}}
\newcommand{\ff}[1]{\mathfrak{#1}}
\newcommand{\inner}[2]{\left\langle {#1},{#2} \right\rangle}

\begin{document}

\title[Universally Symmetric Norming Operators are Compact]{Universally Symmetric Norming Operators are Compact}

\author[Satish~K.~Pandey]{Satish K.~Pandey}
\address{Faculty of Mathematics\\
Technion - Israel Institute of Technology\\
Haifa\; 3200003\\
Israel}
\email{satishpandey@campus.technion.ac.il}
\urladdr{\href{http://noncommutative.space/}
{\url{http://noncommutative.space/}}}
\date{May 21, 2020}

\begin{abstract} 
We study a specific family of symmetric norms on the algebra 
$\bh$ of operators on a separable infinite-dimensional Hilbert space. With respect to each symmetric norm in this family the identity operator fails to attain its norm. Using this, we generalize one of the main results from \cite{SP}; the hypothesis is relaxed, and consequently, the family of symmetric norms for which the result holds is extended.

We introduce and study the concepts of ``universally symmetric norming operators'' and ``universally absolutely symmetric norming operators'' on a separable Hilbert space. These refer to the operators that are, respectively, norming and absolutely norming, with respect to every symmetric norm on $\cl B(\cl H)$. We establish a characterization theorem for such operators and prove that these classes are identical, and that they coincide with the class of compact operators. In particular, we provide an alternative characterization of compact operators on a separable infinite-dimensional Hilbert space.
\end{abstract}

\maketitle

\tableofcontents

\section{Introduction} 

Throughout this paper we shall consider Hilbert spaces over the field $\mathbb{C}$ of complex numbers. 
A bounded linear transformation, henceforth called an ``operator'', $T:\mathcal{H}\to\mathcal{K}$ between two Hilbert spaces is said to be \emph{norming} or \emph{norm attaining} if there is an element $x\in\mathcal{H}$ with $\|x\|=1$ such that $\|T\|=\|Tx\|$, where $\|T\| = \text{sup}\{\|Tx\|_{\cl{K}}: x\in \cl{H}, \|x\|_{\cl{H}} \leqslant 1\}$ is the usual operator norm on the Banach space $\cl B(\cl H, \cl K)$ of operators from $\cl H$ to $\cl K$. We say that $T\in \mathcal{B}(\mathcal{H},\mathcal{K})$ is \emph{absolutely norming} if for every nontrivial closed subspace $\cl M$ of $\cl H$, $T|_\cl M$ is norming. We let $\mathcal N(\mathcal{H},\mathcal{K})$ and $\mathcal{AN}(\mathcal{H},\mathcal{K})$ respectively denote the sets of norming and absolutely norming operators in $\mathcal{B}(\mathcal{H},\mathcal{K})$. Throughout this exposition, the term ``ideal'' will always mean a two-sided ideal. Let us recall the following definition. 
\begin{defn}\label{defn:symmetric-norm}
Let $\ff{I}$ be an ideal of the algebra $\cl{B}(\cl{H})$ of operators on a separable infinite-dimensional Hilbert space $\cl H$. A \textit{symmetric norm} on $\ff{I}$ is defined to be a function $\|\cdot\|_s:\ff{I}\to [0,\infty)$ which satisfies the following conditions:
\begin{enumerate}
\item  $\|\cdot\|_s$ is a norm;
\item $\|X\|_s=\|X\|$ for every rank-one operator $X\in \ff{I}$ (crossnorm property); and
\item $\|AXB\|_s\leq\|A\|\|X\|_s\|B\|$ for every $X\in \ff{I}$ and for every pair $A,B$ of operators in $\cl{B}(\cl{H})$ (uniformity);
\end{enumerate}
where $\|\cdot\|$ is the usual operator norm. 
\end{defn}
If we consider the ideal $\ff I$ to be the full algebra $\cl B(\cl H)$ itself, then $\|\cdot\|_s$ is said to be a symmetric norm on $\cl B(\cl H)$. Moreover, the usual operator norm on any ideal $\ff{I}$ of $\cl{B}(\cl{H})$ is a symmetric norm, and every symmetric norm on $\cl B(\cl H)$ is topologically equivalent to the usual operator norm. 

\subsection{Background and motivation}
The class of norming operators on Hilbert spaces have been extensively studied and there is a plethora of information on these operators; see, for instance, \cite{AcRu02, AcRu98, AcAgPa, Acosta, Aguirre, Partington, Scha2, Scha, Shkarin} and references therein. 
The class of absolutely norming operators, however, was introduced recently in \cite{CN}. In \cite{VpSp}, we established the following spectral characterization theorem for these operators.

\begin{thm}\cite[Theorem 5.1]{VpSp}
Let $T\in \cl B(\cl H, \cl K)$ and $T=U|T|$ its polar decomposition. Then $T$ is absolutely norming, that is, $T\in \cl{AN}(\cl H, \cl K)$ if and only if $|T|$ is of the form $|T|=\alpha I + K + F$, where $\alpha \geq 0$, $K$ is a positive compact operator, and $F$ is a self-adjoint finite-rank operator.
\end{thm}

The above result served to be the first hint to a more general situation. Suppose $\cl B(\cl H, \cl K)$ is equipped with a symmetric norm\footnote{By abuse of terminology, we continue to use the term ``symmetric'' to refer to those norms on $\cl B(\cl H, \cl K)$ which are symmetric in the sense of Definition \ref{defn:symmetric-norm} whenever $\cl H$ is separable and $\cl K=\cl H$.} equivalent to the usual operator norm. Then what does it mean for an operator $T\in \cl B(\cl H, \cl K)$ to be norming and absolutely norming in this setting? What about characterizing these operators?  

In \cite{SP}, we continued the study of absolutely norming operators in this more general setting where $\cl B(\cl H, \cl K)$ was equipped with one of the following three (families of) symmetric norms: the 
Ky Fan $k$-norm(s), the weighted Ky Fan $\pi, k$-norm(s), and 
the $(p,k)$-singular norm(s). We defined the notion of absolutely norming operators in each of these cases and thereafter characterized the set of these operators with respect to each of these three (families of) norms; see \cite[Theorems 4.22, 5.12 and 5.13]{SP}. 
This detailed study of several particular symmetric norms provided us with the insight of \emph{how} to extend the concept of ``norming'' and ``absolutely norming'' from a specific symmetric norm to an arbitrary symmetric norm that is equivalent to the usual operator norm; earlier definitions used the intrinsic nature of each symmetric norm in question. (See Definition \ref{Phi-and-abs-phi-Norming}; the attention has been restricted to the algebra $\cl B(\cl H)$ of operators on a separable infinite-dimensional Hilbert space $\cl H$, and the theory of symmetrically-normed ideals and their Banach space duals play a crucial part in arriving at this definition.)

The subsequent discussion in \cite{SP} involves positive operators of the form of a  nonnegative scalar multiple of the identity plus a positive compact plus a self-adjoint finite-rank. It is not clear, \emph{a priori}, if the operators of this form are absolutely norming with respect to \emph{every} symmetric norm on $\cl B(\cl H)$. It turns out that there exists a symmetric norm on $\cl B(\ell^2)$ with respect to which the identity operator does not attain its norm;
see Theorem \ref{Identity-nonnorming} \cite[Proposition $1.3$]{SP}.

\subsection{Overview of this paper}
The proof of the result presented in \cite{SP} --- that there exists a symmetric norm on $\cl B(\ell^2)$ such that the identity operator does not attain its norm --- in fact, presents a constructive method to produce, not merely one, but an infinite family of symmetric norms on $\cl B(\ell^2)$ with respect to each of which the identity operator does not attain its norm. 

In this manuscript, we formally introduce this family of symmetric norms and study them in detail (see Section \ref{sec:norms-that-are-not-attained-by-the-identity}).
This leads us to Theorem 
\ref{Identity-nonnorming-for-all-nonconstantsequences} which
generalizes the above mentioned result from \cite{SP}; the hypothesis is relaxed, and consequently, the family of symmetric norms for which the result holds is extended.  

We then introduce and study the notion of ``universally symmetric norming operators'' or USN operators
(see Definition \ref{Universally-symmetric-Norming})
and ``universally absolutely symmetric norming operators'' or
UASN operators (see Definition \ref{Universally-absolutely-symmetric-Norming}) 
on a separable infinite-dimensional Hilbert space. 
These refer to the operators that are, respectively, norming and absolutely norming, with respect to \emph{every} symmetric norm.  
It is known (see Theorem \ref{compacts-are-absolutely-symmetric-norming}) that a compact operator in $\cl B(\cl H)$ is universally absolutely symmetric norming, and hence universally symmetric norming. This renders compacts as prototypical examples of such operators. So, we have 
$$
\text{compact operators} \subseteq \text{UASN operators} \subseteq \text{USN operators}.
$$
It would be desirable to know whether a USN operator is compact. In Section \ref{Characterization-Universally-Symmetrically-Norming} we answer this question affirmatively. The following is the main result of this section which essentially states that an operator in $\cl B(\cl H)$ is universally symmetric norming if and only if it is compact.

\begin{customtheorem}{\ref{complete-characterization-arbitrary}}
Let $T\in \cl B(\cl{H})$ and let $\Phi_1$ denote the maximal s.n.function. Then the following statements are equivalent.
\begin{enumerate}
\item $T\in \cl B_0(\cl H)$.
\item $T$ is universally absolutely symmetric norming, that is,
 $T\in\cl{AN}_{\Phi^*}(\cl H)$ for every s.n.function $\Phi$ equivalent to $\Phi_1$.
\item $T$ is universally symmetric norming, that is, $T\in\cl{N}_{\Phi^*}(\cl H)$ for every s.n.function $\Phi$ equivalent to $\Phi_1$.
\end{enumerate}
\end{customtheorem}

 We hence establish a characterization theorem for such operators on $\cl B(\cl H)$. In particular, this result provides an alternative characterization theorem for compact operators on a separable Hilbert space.


\begin{ack}
The author is grateful to his advisor, Vern I. Paulsen, for his suggestions and feedback during the writing of this paper. He would also like to thank Laurent W. Marcoux and Sanat Upadhyay for fruitful conversations about the content of this paper.  
\end{ack}

\section{Preliminaries}\label{Preliminaries}
In this section we recall some notions and results concerning the ideal structure of the algebra of all bounded linear operators acting on a separable Hilbert space. Also, since this work is a continuation of \cite{SP}, this section essentially builds upon, and hence requires, the preliminary section of \cite{SP} where the concept of singular values is extended from compact operators to \emph{any} operator; we refer the reader to \cite[Section $2$]{SP} and \cite{GK}.

\begin{notation}
Consider the algebra $\cl B(\cl H)$ of operators on a separable Hilbert space $\cl H$.
We let $\cl{B}_{00}(\cl{H})$, respectively, $\cl B_0(\cl{H})$ denote the set of all finite-rank operators on $\cl{H}$, respectively, the set of compacts. We use $\cl B_1(\cl{H})$ to denote the trace class operators, with the trace norm $\|\cdot\|_1$.
By $c_0$ we denote the space of all convergent sequences of real numbers with limit $0$ and we let $c_{00}\subseteq c_0$ be the linear subspace of  $c_0$ consisting of all sequences with a finite number of nonzero terms.
 The positive cone of $c_{00}$ is denoted by $c_{00}^+$ and we use $c_{00}^*\subseteq c_{00}^+$ to denote the cone of all nonincreasing nonnegative sequences from $c_{00}$.
\end{notation}

\begin{defn}
An ideal $\ff S$ of the algebra $\cl B(\cl H)$ is said
to be a \emph{symmetrically-normed ideal} (abbreviated an \emph{s.n.ideal}) of $\cl B(\cl H)$ if on it there is
defined a symmetric norm $\|\cdot\|_{\cl{\ff S}}$ which makes $\ff S$ a Banach space (that is, $\ff S$ is complete in the
metric given by this norm). 
We say that two ideals $\ff S_I$ and $\ff S_{II}$ \textit{coincide elementwise} if $\ff S_I$ and $\ff S_{II}$ consist of the same elements.
\end{defn}

\begin{defn}\cite[Chapter 3, Page $71$]{GK}
A function $\Phi:c_{00}\to [0,\infty)$ is said to be \emph{symmetric norming function} (or simply an \emph{s.n.function}) if it satisfies the following properties:
 \begin{enumerate}
 \item[(i)]$\Phi(\xi)\geq 0$ for every $\xi:=(\xi_j)_{j\in \bb N}\in c_{00}$;
\item[(ii)]$\Phi(\xi)= 0 \iff \xi=0$; 
\item[(iii)]$\Phi(\alpha \xi)=|\alpha| \Phi(\xi)$ for every $\xi\in c_{00}$ and for every scalar $\alpha\in \bb R$; and
\item[(iv)]$\Phi(\xi+\psi)\leq \Phi(\xi)+\Phi(\psi)$ for every pair $\xi,\psi$ of sequences in $c_{00}.$
\item[(v)]$\Phi((\xi_1,\xi_2,...,\xi_n,0,0,...))=\Phi((|\xi_{j_1}|,|\xi_{j_2}|,...,|\xi_{j_n}|,0,0,...))$ for every $\xi\in c_{00}$ and for every $n\in \bb N$, where $j_1,j_2,...,j_n$ is any permutation of the integers $1,2,...,n$.
\item[(vi)] $\Phi((1,0,0,...))=1.$
\end{enumerate}
\end{defn}

\begin{defn}\cite[Chapter 3, Page $76$]{GK}
Let $\Phi$ and $\Psi$ be two s.n.functions. $\Phi$ and $\Psi$ are said to be \emph{equivalent} if 
\[
\sup_{\xi\in c_{00}}\frac{\Phi(\xi)}{\Psi(\xi)}<\infty \, \, \text{ and }\, \, \sup_{\xi\in c_{00}} \frac{\Psi(\xi)}{\Phi(\xi)}<\infty.
\]
We say that $\Phi\leq \Psi$ if for every every $\xi \in c_{00}$, we have $\Phi(\xi)\leq \Psi(\xi)$.
\end{defn}
\begin{remark}
A moment's thought will convince the readers that an s.n.function can be uniquely defined by its values on the cone $c_{00}^*$.
Here the \emph{minimal} s.n.function $\Phi_\infty:c_{00}^*\to [0,\infty)$ is defined by $
\Phi_\infty(\xi)=\xi_1 \text{ for every } \xi=(\xi_j)_j\in c_{00}^*$
and the \emph{maximal} s.n.function $\Phi_1:c_{00}^*\to [0,\infty)$ is defined by $
\Phi_1(\xi)=\sum_{j}\xi_j \text{ for every } \xi=(\xi_j)_j\in c_{00}^*.$ For any s.n.function $\Phi$, we have $\Phi_\infty\leq \Phi \leq \Phi_1$  (see \cite[Chapter 3, Section 3, Relation $3.12$, Page $76$]{GK}).
\end{remark}
 
\begin{defn}
Let $\Phi$ be an s.n.function defined on $c_{00}^*$. Then the function given by the formula  
$$
\Phi^*(\eta)=\max \left\{\sum_{j}\eta_j\xi_j: \xi\in c_{00}^*, \Phi(\xi)=1 \right\} \text{ for every }\eta \in c_{00}^*,
$$ 
is defined to be the \emph{adjoint of the function} $\Phi$.
\end{defn}
\begin{remark}
That $\Phi^*$ is an s.n.function is a trivial observation. Also, the adjoint of $\Phi^*$ is $\Phi$. In particular, the minimal and maximal s.n.functions are the adjoint of each other, that is, 
$\Phi_1^*=\Phi_{\infty}$ and $\Phi_\infty^*=\Phi_1.$ Therefore, when an s.n.function is equivalent to the maximal (minimal) one, its adjoint is equivalent to the minimal (maximal) one.
\end{remark}

\begin{remark}
It is evident that every s.n.ideal gives rise to an s.n.function. Conversely, to every s.n.function $\Phi$ we associate an s.n.ideal $\ff S_{\Phi}$, which is referred to as the s.n.ideal generated by the s.n.function $\Phi$. For a detailed exposition of the construction of the s.n.ideal from an s.n.function we refer the reader to Gohberg and Krein's text \cite[Chapter $3$]{GK} which elaborately discusses the theory of s.n.ideals. An abridged outline of this construction has also been discussed in \cite[Section $6$, Notation $6.4$]{SP}. Since we will be dealing with this theory extensively, we have attempted to duplicate their notation wherever possible.

  If $\Phi,\Psi$ are s.n.functions and $\ff S_\Phi,\ff S_\Psi$ are the s.n.ideals generated by these s.n.functions respectively, then $\ff S_\Phi$ and $\ff S_\Psi$ coincide elementwise if and only if $\Phi$ and $\Psi$ are equivalent. In particular, if $\Phi$ is an s.n.function equivalent to  $\Phi_1$, then  $\ff S_{\Phi}$ and $\cl B_1(\cl H)$ coincide elementwise and when $\Phi$ is equivalent to $\Phi_\infty$, $\ff S_{\Phi}$ and $\cl B_0(\cl H)$ coincide elementwise. 
Two more notations are in order. The Banach space dual of a Banach space $X$ is denoted by $X^*$ in the sequel. To indicate that Banach spaces $X$ and $Y$ are
isometrically isomorphic, we write $X\cong Y$ or $Y\cong X$ isometrically.
\end{remark}
 
We conclude this section by an often needed piece of folklore from \cite{GK}.
\begin{prop}\cite[Chapter 3, Theorem 12.4]{GK}\label{Gohb-Krein}
If $\Phi$ is an arbitrary s.n.function equivalent to the maximal s.n.function, then the general form of a continuous linear functional $f$ on the separable space $\ff S_\Phi$ is given by $f(X)=\tr(AX)$ for some $A\in \cl B(\cl H)$ and 
 $$\|f\|:=\sup\{|\tr(AX)|:X\in \ff S_\Phi, \|X\|_\Phi\leq 1\}=\|A\|_{\Phi^*}.$$
Thus, the Banach space dual $\ff S_\Phi^*$ is isometrically isomorphic to $(\cl B(\cl H),\|\cdot\|_{\Phi^*})$, that is, 
$$\ff S_\Phi^*\cong (\cl B(\cl H),\|\cdot\|_{\Phi^*}).$$ \end{prop}

\section{Norms that are not attained by the identity}\label{sec:norms-that-are-not-attained-by-the-identity}

In this section, we revisit a result from \cite{SP} and use it to formally introduce a family of symmetric norms on $\cl B(\cl H)$ with respect to each of which the identity operator is rendered nonnorming. Thereafter, in Subsection \ref{generalization-of-a-result-from-previous-paper}, we 
improve upon this above mentioned result by extending 
it to a larger family of symmetric norms. 

\subsection{Symmetric norming and absolutely symmetric norming operators}
Given an arbitrary s.n.function $\Phi$ that is equivalent to the maximal s.n.function, we now recall the definition of operators in $\cl B(\cl H)$ that attain their $\Phi^*$-norm.

\begin{defn}\cite[Definitions $6.11, 6.14$]{SP}\label{Phi-and-abs-phi-Norming}
Let $\Phi$ be an s.n.function equivalent to the maximal s.n.function. An operator $T\in (\cl B(\cl H),\|\cdot\|_{\Phi^*})$ is said to be \emph{$\Phi^*$-norming} or \emph{symmetric norming with respect to the symmetric norm $\|\cdot\|_{\Phi^*}$}  if there exists an operator $K\in \ff S_\Phi=\cl B_1(\cl H)$ with $\|K\|_{\Phi}=1$ such that $|\tr(TK)|=\|T\|_{\Phi^*}.$ We say that $T\in (\cl B(\cl H),\|\cdot\|_{\Phi^*})$ is \emph{absolutely $\Phi^*$-norming} or \emph{absolutely symmetric norming with respect to the symmetric norm $\|\cdot\|_{\Phi^*}$} if for every nontrivial closed subspace $\cl M$ of $\cl H$, $TP_{\cl M}\in \cl B(\cl H)$ is $\Phi^*$-norming (here $P_\cl M$ is the orthogonal projection onto $\cl M$).

We let $\cl N_{\Phi^*}(\cl H)$ and $\cl{AN}_{\Phi^*}(\cl H)$ respectively denote the set of $\Phi^*$-norming and absolutely $\Phi^*$-norming operators in $\cl B(\cl H)$. Needless to mention, every absolutely $\Phi^*$-norming operator is $\Phi^*$-norming, that is, $\cl{AN}_{\Phi^*}(\cl H)\subseteq \cl{N}_{\Phi^*}(\cl H)$.
\end{defn}

With the above definitions established to guide the way, we prove and collect certain fundamental results concerning symmetric norming and absolutely symmetric norming operators in $\bh$. We begin by recalling the following result. 

\begin{thm}\cite[Theorem $6.17$]{SP}\label{compacts-are-absolutely-symmetric-norming}
Let $\Phi$ be an arbitrary s.n.function equivalent to the maximal s.n.function. If $T$ is a compact operator, then $T\in \cl{AN}_{\Phi^*}(\cl H),$ that is, $\cl B_0(\cl H)\subseteq \cl{AN}_{\Phi^*}(\cl H).$
\end{thm}
Following the well established precedent, we use $s_j(T)$ to denote the $j$-th singular value (or singular number or s-number) of $T\in \cl B(\cl H)$.
The following proposition allows us to concentrate on the positive operators that are symmetric norming. The technique of the proof might be elementary but since we do not have a reference for the exact statement, a complete proof is provided.
\begin{prop}\label{T-iff-|T|-phi-star}
Let $\Phi$ be an s.n.function equivalent to the maximal s.n.function. Then $T\in \cl N_{\Phi^*}(\cl H)$ 
if and only if $|T|\in \cl N_{\Phi^*}(\cl H).$
\end{prop}

\begin{proof}
We first assume that $T\in \cl N_{\Phi^*}(\cl H)$ and observe that $\|T\|_{\Phi^*}=\|\,|T|\,\|_{\Phi^*}$ since for each $j$, $s_j(T)=s_j(|T|)$. Then there exists $K\in \cl B_1(\cl H)$ with $\|K\|_\Phi=1$ such that $\|T\|_{\Phi^*}=|\tr(TK)|$. If $T=U|T|$ is the polar decomposition of $T$, then $$
\|\,|T|\,\|_{\Phi^*}=\|T\|_{\Phi^*}=|\tr(TK)|=|\tr(U|T|K)|=|\tr(|T|KU)|,
$$ 
where $KU\in \cl B_1(\cl H)$ with $\|KU\|_\Phi=\|IKU\|_\Phi\leq \|I\|\|K\|_\Phi\|U\|=\|K\|_\Phi=1$. In fact, $\|KU\|_\Phi=1$; for if not, then the operator $S:=KU/\|KU\|_{\Phi}\in \cl B_1(\cl H)$ satisfies $\|S\|_\Phi=1$ and yields
\begin{align*}
|\tr(|T|S)|&=\left|\tr\left(\frac{|T|KU}{\|KU\|_{\Phi}}\right)\right|=\frac{1}{\|KU\|_{\Phi}}|\tr(|T|KU)|>|\tr(|T|KU)|=\|\,|T|\,\|_{\Phi^*},
\end{align*}
which contradicts the fact that the supremum of the set $$\{|\tr(|T|X)|:X\in \cl B_1(\cl H),\|X\|_\Phi\leq 1\}$$ is attained at $KU$. This shows that $|T|\in \cl N_{\Phi^*}(\cl H).$ 

Conversely, if $|T|\in \cl N_{\Phi^*}(\cl H)$, then by replacing $T$ with $|T|$ in the above argument using $|T|=U^*T$, we can prove the existence of  $\hat K\in \cl B_1(\cl H)$ with $\|\hat K\|_\Phi=1$ such that $\|T\|_{\Phi^*}=|\tr(T\hat KU^*)|$ where $\hat KU^*\in \cl B_1(\cl H)$ with $\|\hat KU*\|_\Phi\leq1$. It can then be shown that $\|\hat KU^*\|_\Phi=1$ and the result follows.
\end{proof}

We need one more result concerning the computation of the symmetric norm of an operator.

\begin{prop}\label{PhiNorming-alternative-II}
Let $\Phi$ be an s.n.function equivalent to the maximal s.n.function and let $T\in \cl B(\cl H)$. Then 
$$
\|T\|_{\Phi^*}=\sup\left\{\sum_{j}s_j(T)s_j(K):K\in \cl B_1(\cl H), K=\text{\emph{diag}}\{s_j(K)\}_{j}, \|K\|_\Phi=1\right\}
$$
\end{prop}
\begin{proof}
Since $\Phi$ is equivalent to the maximal s.n.function,  we know that $\ff{S}_{\Phi}^*\cong  (\cl B(\cl H), \|\cdot\|_{\Phi^*})$ isometrically
and by Definition \ref{Phi-and-abs-phi-Norming} the $\|\cdot\|_{\Phi_\pi^*}$ norm for any operator $T\in \cl B(\cl H)$ is given by 
$\|T\|_{\Phi^*}= \sup\{|\tr(TK)|:K\in \ff S_{\Phi},\, \|K\|_{\Phi}= 1\}.$ But the ideal $\cl B_1(\cl H)$ and $\ff S_{\Phi_\pi}$ coincide elementwise and hence 
$\|T\|_{\Phi^*}= \sup\{|\tr(TK)|:K\in \cl B_1(\cl H),\, \|K\|_{\Phi}= 1\}.$

First we set $\alpha:=\sup\{|\tr(TK)|:K\in \cl B_1(\cl H),\, \|K\|_{\Phi}= 1\}$ and $\beta:=\sup\{\sum_{j}s_j(T)s_j(K):K\in \cl B_1(\cl H),\, \|K\|_{\Phi}= 1\}$, and thereafter we claim that $\alpha=\beta.$
That $\alpha\leq \beta$ is a trivial observation since $|\tr(TK)|\leq \sum_{j}s_j(TK)\leq \sum_js_j(T)s_j(K).$ To see $\beta \leq \alpha$, let us choose an operator $K\in \cl B_1(\cl H)$ with $\|K\|_{\Phi}=1$. An easy computation yields 
\begin{align*}
\sum_{j}s_j(T)s_j(K)&=
\inner{\begin{bmatrix}
s_1(T)\\
\vdots\\
s_j(T)\\
\vdots
\end{bmatrix}}{\begin{bmatrix}
s_1(K)\\
\vdots\\
s_j(K)\\
\vdots
\end{bmatrix}}\\
&\leq \Phi^*\left(\begin{bmatrix}
s_1(T)\\
\vdots\\
s_j(T)\\
\vdots
\end{bmatrix}\right)\Phi\left(\begin{bmatrix}
s_1(K)\\
\vdots\\
s_j(K)\\
\vdots
\end{bmatrix}\right)\\
&=\Phi^*\left(\begin{bmatrix}
s_1(T)\\
\vdots\\
s_j(T)\\
\vdots
\end{bmatrix}\right)=\|T\|_{\Phi^*}\\
&=\sup\{|\tr(TK)|:K\in \cl B_1(\cl H),\, \|K\|_{\Phi}= 1\}=\alpha.
\end{align*}
It then follows that $\beta \leq \alpha$ and this proves our first claim.

We next let $\gamma:=\sup\{\sum_{j}s_j(T)s_j(K):K\in \cl B_1(\cl H),\,K=\diag\{s_j(K)\}, \,\|K\|_{\Phi}$ $= 1\}$ and prove that $\gamma=\beta$. That $\gamma \leq \beta$ is obvious. To prove $\beta\leq \gamma$, we choose an operator $K\in \cl B_1(\cl H)$ with $\|K\|_{\Phi}=1$ and define $$
\tilde K:=\begin{pmatrix}
s_1(K)\\
&s_2(K)&&\text{\huge 0}\\
&&\ddots&&\\
&\text{\huge 0}& & s_j(K)&\\
&&&&\ddots
\end{pmatrix}.
$$
Notice that for every $j$, we have $s_j(\tilde K)=s_j(K)$ which implies that $\|\tilde K\|_\Phi=\|K\|_\Phi=1$. Even more, $\tilde K\in \cl B_1(\cl H)$ and hence $\sum_{j}s_j(T)s_j(K)=\sum_{j}s_j(T)s_j(\tilde K)$. But since
\begin{align*}
\sum_{j}s_j(T)s_j(\tilde K)\leq \sup\left\{\sum_{j}s_j(T)s_j(K):K\in \cl B_1(\cl H),\,K=\text{diag}\{s_j(K)\},\, \|K\|_{\Phi}= 1\right\},
\end{align*} 
it follows that $\beta \leq \gamma$ which establishes our second claim. From the above two observations we conclude that $\alpha=\gamma$, and consequently the assertion is proved.
\end{proof}

\subsection{Norm(s) that are not attained by the identity}

We now recall the following result which states that there exists a symmetric norm on $\cl B(\cl H)$ with respect to which the identity operator does not attain its norm.

\begin{thm}\cite[Proposition $1.3$]{SP}\label{Identity-nonnorming}
There exists a symmetric norm $\|\cdot\|_{\Phi_\pi^*}$ on $\cl B(\ell^2)$ such that $I\notin \cl N_{\Phi_\pi^*}(\ell^2)$.
\end{thm}

There is more to this theorem than meets the eye; its proof is constructive and illustrates an elegant technique of producing
(a family of) symmetric norms on $\cl B(\cl H)$ with respect to each of which the identity operator does not attain its norm. More precisely, the proof demonstrates a well-defined family of s.n.functions --- henceforth referred to as ``s.n.functions 
affiliated to strictly decreasing weights'' --- which naturally 
generate such symmetric norms. In what follows we formally introduce this family of s.n.functions.

Let $\widehat \Pi$ denote the set of all strictly decreasing convergent sequences of positive numbers with their first term equal to $1$ and positive limit, that is, 
$$\widehat \Pi=\{\pi:=(\pi_n)_{n\in \bb N}: \pi_1=1,\, \lim\pi_n>0, \text{ and } \pi_k>\pi_{k+1} \text{ for every }k\in \bb N\}.$$ 
For each $\pi\in \widehat \Pi$, let $\Phi_\pi$ denote the symmetric norming function defined by $\Phi_\pi(\xi_1,\xi_2,...)=\sum_j\pi_j\xi_j$ and observe that $\Phi_\pi$ is equivalent to the maximal s.n.function $\Phi_1$.
Theorem \ref{Identity-nonnorming} essentially proves that $I\notin \cl N_{\Phi_\pi^*}(\cl \ell^2)$ whenever $\Phi_\pi$ belongs to the family $\{\Phi_\pi:\pi\in \widehat \Pi\}$ of s.n.functions affiliated to strictly decreasing weights. 

\subsection{A few more norms that are not attained by the identity}\label{generalization-of-a-result-from-previous-paper}
We construct more symmetric norms on $\bh$ with respect to which the identity operator does not attain its norm. 
This subsection aims at
extending Theorem \ref{Identity-nonnorming} to a larger family
of symmetric norms, and thus generalizing it. 

Observe that the proof of Theorem \ref{Identity-nonnorming}, in its construction of symmetric norms, assumes that the sequence $\pi$ belongs to $\widehat{\Pi}$ and is thus strictly decreasing.
We claim that the requirement for $\pi$ to be a strictly decreasing sequence is not necessary for the result to hold, as long as there exists a natural number $N$ such that $\pi_N>\pi_{N+1}$. Let us use $\Pi$ to denote the set of all nonincreasing sequences of positive numbers with their first term equal to $1$ and positive limit (so that $\widehat{\Pi}\subseteq \Pi$)\footnote{For the convenience of the readers, it would perhaps be worth to recall that we have used $\Pi$, in \cite{SP}, to denote the set of all nonincreasing sequences of positive numbers with their first term equal to $1$. However, in this manuscript, we use $\Pi$ to denote the set of all nonincreasing sequences of positive numbers with their first term equal to $1$ and positive limit. This abuse of notation deemed necessary to avoid too many symbols.}. 
Thus, if $\pi\in \Pi$ is not the constant sequence ${\bf 1}=(1,1,...)$, then there exists a natural number $N$ such that $\pi_N>\pi_{N+1}$, and consequently, 
our claim amounts to showing that the adjoint of the
s.n.function defined via $\pi$ generates a symmetric norm
on $\cl B(\ell^2)$ with
respect to which the identity operator does not attain its norm.
 
We conclude this section by proving the following result which, in effect, extends the preceding theorem to every symmetric norm $\Phi_\pi$ in the family $\{\Phi_\pi:\pi\in \Pi\}$ of s.n.functions except when $\pi\in \Pi$ is the constant sequence ${\bf 1}=(1,1,...)$; earlier, the result was shown to hold for the family $\{\Phi_\pi:\pi\in \widehat{\Pi}\}$ of s.n.functions. 

\begin{thm}\label{Identity-nonnorming-for-all-nonconstantsequences}
Let $\Pi$ be the set of all nonincreasing convergent sequence of positive numbers with their first term equal to $1$ and positive limit, that is, $$\Pi=\{\pi:=(\pi_n)_{n\in \bb N}: \pi_1=1,\, \lim_{n}\pi_n>0, \text{ and } \pi_k\geq\pi_{k+1} \text{ for each }k\in \bb N\},$$ and consider the subset $\Pi\setminus \{{\bf 1}\}$ of $\Pi$ consisting of all nonincreasing convergent sequence of positive numbers with their first term equal to $1$ and positive limit except the constant sequence ${\bf 1}$. For each $\pi\in \Pi\setminus \{{\bf 1}\}$, let $\Phi_\pi$ denote the symmetric norming function defined by $\Phi_\pi(\xi_1,\xi_2,...)=\sum_j\pi_j\xi_j.$ Then \begin{enumerate}
\item $\Phi_\pi$ is equivalent to the maximal s.n.function $\Phi_1$; and 
\item for every $\pi\in \Pi\setminus\{{\bf 1}\}$, $I\notin \cl N_{\Phi_{\pi}^*}(\ell^2)$.
\end{enumerate}
Alternatively, $I\notin \cl N_{\Phi_{\pi}^*}(\ell^2)$ for every $\Phi_\pi$ that belongs to the family $\{\Phi_{\pi}:\pi\in \Pi\setminus\{{\bf 1}\}\}$ of s.n.functions.  
\end{thm}

\begin{proof}
That each s.n.function from the family $\{\Phi_{\pi}:\pi\in \Pi\setminus\{{\bf 1}\}\}$ of s.n.functions is equivalent to the maximal s.n.function $\Phi_1$ is, now, a trivial observation. 

The proof of the second claim is almost along the lines of the proof of Theorem \ref{Identity-nonnorming}. Let $\pi=(\pi_n)_{n\in \bb N}\in \Pi\setminus \{{\bf 1}\}$ and let $\Phi_{\pi}$ be the s.n.function generated by $\pi$, that is, $\Phi_{\pi}(\xi_1,\xi_2,...)=\sum_j\pi_j\xi_j$. Now, contrapositively assume that $I\in \cl N_{\Phi_\pi^*}$, then the supremum, 
 $$\sup\left\{\sum_{j}s_j(K):K\in \cl B_1(\ell^2),\, K=\text{diag}\{s_1(K),s_2(K),...\},\,\|K\|_{\Phi_\pi}=1\right\},$$ 
 is attained, that is, 
  there exists $K_0=\text{diag}\{s_1(K_0),s_2(K_0),...\}\in \cl B_1(\ell^2)$ with $\sum_{j}\pi_js_j(K_0)=1$ such that $\|I\|_{\Phi_\pi^*}=\sum_js_j(K_0).$ 
  We will prove the existence of an operator $\tilde K\in \cl B_1(\ell^2), \,\|\tilde K\|_{\Phi_\pi}=1$ of the form $\tilde K=\text{diag}\{s_1(\tilde K),s_2(\tilde K),...\}$ such that $\sum_is_i(\tilde K)> \sum_js_j(K_0)$. To this end, since $\pi\in \Pi\setminus \{{\bf 1}\}$, there exists a natural number $M$ such that $\pi_M>\pi_{M+1}$. Set 
$$
\lambda=\frac{\pi_M}{\pi_{M+1}},
$$
and choose $\epsilon >0$ such that 
$$
s_M(K_0)-\epsilon=s_{M+1}(K_0)+\lambda \epsilon.
$$
(Of course, there is only one such $\epsilon$.)
Now define $\tilde K$ to be the diagonal operator given by 
$$
\tilde K:=\begin{bmatrix}
s_1(K_0)\\
&\ddots\\
&&s_{M-1}(K_0)\\
&&&s_M(K_0)-\epsilon\\
&&&&s_{M+1}(K_0)+\lambda \epsilon\\
&&&&&s_{M+1}(K_0)\\
&&&&&&\ddots\\
\end{bmatrix}.
$$
Before proceeding further, notice that the singular numbers of the above defined opearator $\tilde K$ are precisely the diagonal elements, and that the equation preceding the definition of $\tilde K$ guarantees that these s-numbers are arranged in a nonincreasing manner on the diagonal. Next we observe that
\begin{align*}
-\pi_M\epsilon+\pi_{M+1}\lambda \epsilon=-\pi_M\epsilon+\pi_{M+1}\frac{\pi_M}{\pi_{M+1}} \epsilon=-\pi_M\epsilon+-\pi_M\epsilon=0,
\end{align*}
and hence \begin{align*}
 \pi_M(s_M(K_0)-\epsilon)+\pi_{M+1}(s_{M+1}(K_0)+\lambda \epsilon)=\pi_Ms_M(K_0)+\pi_{M+1}s_{M+1}(K_0).
 \end{align*}
 This yields 
 \begin{align*}
\sum_j\pi_js_j(\tilde K)=\sum_j\pi_js_j(K_0),\text{ and hence }\|\tilde K\|_{\Phi_\pi}=
\|K_0\|_{\Phi_\pi}=1.
 \end{align*} 
Consequently, we have $\tilde K\in \cl B_1(\ell^2)$. Moreover, $\tilde K$ is of the form $\tilde K=\diag\{s_1(\tilde K), s_2(\tilde K),...\}$.

However, since  $\lambda >1$ and $\epsilon >0$, we have
$$
(s_M(K_0)-\epsilon) + (s_{M+1}(K_0)+\lambda \epsilon)= s_M(K_0) + s_{M+1}(K_0)+(\lambda-1) \epsilon > s_{M}(K_0)+s_{M+1}(K_0),
$$
which allows us to infer that 
\begin{align*}
\sum_js_j(\tilde K)&=\sum_{j=1}^{M-1}s_j(K_0)+(s_M(K_0)-\epsilon) + (s_{M+1}(K_0)+\lambda \epsilon) +\sum_{j>M+1}\pi_js_j(K_0)\\
&>\sum_{j=1}^{M-1}s_j(K_0)+\left(s_{M}(K_0)+s_{M+1}(K_0)\right) +\sum_{j>M+1}\pi_js_j(K_0)\\
&=\sum_js_j(K_0),
\end{align*}
which contradicts the assumption that $\sum_js_j(K_0)$ ($=|\tr(K_0)|$) is the supremum of the set 
$$
\left\{\sum_{j}s_j(K):K\in \cl B_1(\ell^2),\, K=\text{diag}\{s_1(K),s_2(K),...\},\,\|K\|_{\Phi_\pi}=1\right\}.
$$ Since the operator $K_0$ with which we began our discussion is arbitrary, it follows that for any given operator in $\cl B_1(\ell^2)$ with unit norm, one can find another operator in $ \cl B_1(\ell^2)$ with unit norm with trace of larger magnitude and hence the supremum of the above set can never be attained. This shows that the identity operator does not attain its norm.
\end{proof}

\section{Symmetric norming operators affiliated to strictly decreasing weights and their characterization}\label{Characterization-Affiliated-To-Decreasing-Weights}

In this section we return to the study of the family of symmetric norms on $\cl B(\cl H)$ generated by the duals (or adjoints) of s.n.functions from the family $\{\Phi_{\pi}:\pi\in \widehat{\Pi}\}$ of s.n.functions affiliated to strictly decreasing weights, and establish a characterization theorem for  operators in $\cl B(\cl H)$ that are symmetric norming with respect to \emph{every} such symmetric norm. (Recall that for each $\pi\in \widehat \Pi$, $\Phi_\pi$ denotes the s.n.function defined by $\Phi_\pi(\xi_1,\xi_2,...)=\sum_j\pi_j\xi_j$, and that $\Phi_\pi$ is equivalent to the maximal s.n.function $\Phi_1$.) This section also studies the operators in $\cl B(\cl H)$ that are absolutely symmetric norming with respect to \emph{every} symmetric norm in the family and presents a characterization theorem for those as well. It turns out that an operator is symmetric norming with respect to every symmetric norm in the family if and only if it is absolutely symmetric norming with respect to every symmetric norm in the family. This ``characterization theorem'' is the main theorem of this section.

Theorem \ref{Identity-nonnorming} essentially proves that $I\notin \cl N_{\Phi_\pi^*}(\cl H)$ whenever $\pi\in \widehat \Pi$.
We know that for an arbitrary s.n.function $\Phi$ equivalent to $\Phi_1$, we have $\cl N_{\Phi^*}(\cl H)\nsubseteq \cl B_0(\cl H)$. However, it is of interest to know whether $\cl N_{\Phi_\pi^*}(\cl H)\subseteq \cl B_0(\cl H)$ if $\pi\in \widehat \Pi$; for if the answer to this question is affirmative, then Theorem \ref{compacts-are-absolutely-symmetric-norming} would yield $\cl N_{\Phi_\pi^*}(\cl H)= \cl B_0(\cl H)$ for every $\pi\in \widehat \Pi$ (and would thus characterize the $\Phi_{\pi}^*$-norming operators in $\cl B(\cl H)$ for every $\pi\in \widehat \Pi$). By Proposition \ref{T-iff-|T|-phi-star} it suffices to know whether $\cl N_{\Phi_\pi^*}(\cl H)\cap \cl B(\cl H)_+\subseteq \cl B_0(\cl H)$ where $\cl B(\cl H)_+=\{T\in \cl B(\cl H):T\geq 0\}.$ The following lemma and example prove the existence of $\pi\in \widehat \Pi$ such that $\cl N_{\Phi_\pi^*}(\cl H)\nsubseteq \cl B_0(\cl H)$.

\begin{lemma}[\cite{GK}]
If $\Phi_\pi\in \{\Phi_\pi:\pi\in \widehat \Pi\}$, then its adjoint $\Phi_\pi^*$ is given by 
$$
\Phi_\pi^*(\xi)=\sup_{n}\left\{\frac{\sum_{j=1}^n\xi_j}{\sum_{j=1}^n\pi_j}\right\}\, \text{ for every } \xi=(\xi_i)_{i\in \bb N}\in c_{00}^*. 
$$ Moreover, the s.n.function $\Phi_\pi^*$ is equivalent to the minimal s.n.function.
\end{lemma}
For the proof of the above lemma we refer the reader to \cite[Chapter 3, Lemma 15.1, Page 147]{GK}; readers can also see Pages 148-149, and the paragraph preceding Theorem 15.2 of the monograph \cite{GK}.

\begin{exam}\label{violates-the-intuition}
Consider the positive diagonal operator 
$$P=\begin{bmatrix}
 2\\
    & 1+\frac{1}{2} & & & &\text{\huge0}&\\
    & &1+ \frac{1}{3}\\
    & & & 1+\frac{1}{4}\\
    & & & & \ddots&\\
    &\text{\huge0} & & & &1+\frac{1}{n} &\\
    & & &  & & & \ddots
\end{bmatrix} \in \cl{B}( \ell^2),$$
with respect to an orthonormal basis $B=\{v_i:i\in \bb N\}$. Let $\pi=(\pi_n)_{n\in \bb N}$ be a sequence of real numbers defined by $\pi_n:=\frac{1}{2}+\frac{1-1/2}{n}=\frac{n+1}{2n}$. That $\pi\in \widehat \Pi$ is obvious. Consequently, $\Phi_\pi$ is equivalent to the maximal s.n.function $\Phi_1$ and the Banach space dual $\ff S_{\Phi_\pi}^* $ of the s.n.ideal $\ff S_{\Phi_\pi}$ is isometrically isomorphic to $(\cl B(\ell ^2),\|\cdot\|_{\Phi_\pi^*})$, that is, $\ff S_{\Phi_\pi}\cong(\cl B(\ell ^2),\|\cdot\|_{\Phi_\pi^*})$ isometrically. An easy computation yields 
$$
\|P\|_{\Phi_\pi^*}
=\sup_{n}\left\{\frac{\sum_{j=1}^ns_j(P)}{\sum_{j=1}^n\pi_j}\right\}
=\sup_{n}\left\{\frac{ n+(1+1/2+...+1/n)}{\frac{1}{2}(n+(1+1/2+...+1/n))}\right\}
=2.
$$
If we define $K$ to be the diagonal operator given by 
$$
K=\begin{pmatrix}
1\\
&0\\
&&\ddots\\
&& &0\\
&&&&\ddots
\end{pmatrix}\in \cl{B}_1( \ell^2)=\ff S_{\Phi_\pi},
$$
then we have $\|K\|_{\Phi_\pi}=\sum_j\pi_js_j(K)=1$ and $|\tr(PK)|=|\tr(\text{diag}\{2,0,0,...\})|$
$=2=\|P\|_{\Phi_\pi^*}$ which implies that $P\in\cl N_{\Phi_\pi^*}(\cl H)$. However, $P\notin \cl B_0(\ell^2)$. This proves the existence of $\pi\in \widehat \Pi$ such that $\cl N_{\Phi_\pi^*}(\cl H)\nsubseteq \cl B_0(\cl H)$.
\end{exam}

The above example establishes the fact that even for a given $\Phi_\pi$ from the family $\{\Phi_\pi:\pi\in \widehat \Pi\}$ of s.n.functions, it is too much to ask for the set $\cl N_{\Phi_\pi^*}(\cl H)$ to be contained in the compacts. So let us be more modest and ask whether $P\in \cl B(\cl H)$ is compact whenever $P\in \cl N_{\Phi_\pi^*}(\cl H)\cap \cl B(\cl H)_+$ for \emph{every} $\Phi_\pi\in \{\Phi_\pi:\pi\in \widehat \Pi\}.$ The answer to this question is a resounding yes as is stated in the Theorem \ref{mini-characterization-positives}. 

Before we prove this theorem, let us find $\pi\in \widehat \Pi$ such that the positive noncompact operator $P$ of Example \ref{violates-the-intuition} does not belong to $\cl N_{\Phi_{\pi}^*}(\cl H)$. The example which follows illustrates this and hence agrees with the Theorem \ref{mini-characterization-positives}. 

\begin{exam}\label{Intuitive-to-minicharacterization-for-positives}
Let $\pi=(\pi_n)_{n\in \bb N}$  be a sequence defined by $$\pi_n:=\frac{1}{3}+\frac{1-1/3}{n}=\frac{n+2}{3n}.$$ Then $\pi\in \widehat \Pi,\,\,\Phi_\pi\in  \{\Phi_\pi:\pi\in \widehat \Pi\}\text{ and }\ff S_{\Phi_\pi}\cong(\cl B(\ell ^2),\|\cdot\|_{\Phi_\pi^*})$ isometrically. We consider the operator $P$ of Example \ref{violates-the-intuition} and prove that $P\notin \cl N_{\Phi_\pi^*}(\ell^2)$. To show this, we assume that $P\in \cl N_{\Phi_\pi^*}(\ell^2)$, that is, the supremum,
$$\sup\left\{\sum_{j}s_j(P)s_j(K):K\in \cl B_1(\ell^2), K=\text{diag}\{s_1(K),s_2(K),...\}, \|K\|_{\Phi_\pi}=1\right\},$$ is attained, 
and we deduce a contradiction from this assumption. So there exists $K=\text{diag}\{s_1(K),s_2(K),...\}\in \cl B_1(\ell^2)$ with $\|K\|_{\Phi_\pi}=\sum_j\pi_js_j(K)=1$ such that $\|P\|_{\Phi_\pi^*}=|\tr(PK)|=\sum_js_j(P)s_j(K)$. 
Since $K\in \cl B_1(\cl H)\subseteq \cl B_0(\cl H)$, we have $\lim_{j\rightarrow \infty} s_j(K)=0$. This forces the existence of a natural number $M$ such that $s_M(K)>s_{M+1}(K).$ All that remains is to show the existence of an operator $\tilde K\in \cl B_1(\cl H), \,\|\tilde K\|_{\Phi_\pi}=1$ of the form $\tilde K=\text{diag}\{s_1(\tilde K),s_2(\tilde K),...\}$ such that $\sum_js_j(P)s_j(\tilde K)> \sum_js_j(P)s_j(K)$. If we define $$
\lambda:=\frac{\sum_{j=M}^{M+1}\pi_js_j(K)}{\sum_{j=M}^{M+1}\pi_j}=\frac{\pi_Ms_M(K)+\pi_{M+1}s_{M+1}(K)}{\pi_M+\pi_{M+1}}
$$ 
and let $\tilde K$ be the diagonal operator defined by 
$$
\tilde K:=\begin{pmatrix}
s_1(K)\\
&\ddots&&\\
&&s_{M-1}(K)&&\\
&&&\lambda&\\
&&&&\lambda\\
&&&&&s_{M+1}(K)\\
&&&&&&\ddots
\end{pmatrix},
$$
then for every $j$, $s_j(\tilde K)=s_j(K)$ which implies that $\|\tilde K\|_{\Phi_\pi}=\|K\|_{\Phi_\pi}=1$ so that $\tilde K\in \cl B_1(\ell^2)$ and is of the form $\tilde K=\text{diag}\{s_1(\tilde K),s_2(\tilde K),...\}$. We now prove that $\tilde K$ is the required candidate. It is not too hard to see that 
$$
\frac{\pi_M}{\pi_{M+1}}>\frac{s_M(P)}{s_{M+1}(P)},
$$ 
which yields,
$$
\pi_Ms_{M+1}(P)(s_M( K)-s_{M+1}( K))>\pi_{M+1}s_M(P)(s_M( K)-s_{M+1}(K)).
$$
Simplification and rearrangement of terms in the above inequality gives
$$
\left(s_M(P)+s_{M+1}(P)\right)\left[\frac{\pi_Ms_M(K)+\pi_{M+1}s_{M+1}(K)}{\pi_M+\pi_{M+1}}\right]>s_M(P)s_M(K)+s_{M+1}(P)s_{M+1}(K).
$$
But the left hand side of the above inequation is actually 
$s_M(P)s_M(\tilde K)+s_{M+1}(P)s_{M+1}(\tilde K),$ which implies  
$$
s_M(P)s_M(\tilde K)+s_{M+1}(P)s_{M+1}(\tilde K)>s_M(P)s_M(K)+s_{M+1}(P)s_{M+1}(K).
$$
It then immediately follows that $\sum_{j}s_j(P)s_j(\tilde K)>\sum_{j}s_j(P)s_j(K)$ which contradicts the assumption that $\sum_{j}s_j(P)s_j(K)$ is the supremum of the set $$\left\{\sum_{j}s_j(P)s_j(K):K\in \cl B_1(\ell^2), K=\text{diag}\{s_1(K),s_2(K),...\}, \|K\|_{\Phi_\pi}=1\right\}$$ and this is precisely the assertion of our claim.
\end{exam}

\begin{remark}
The working rule of the above example is illuminating. The sequence $\pi=(\pi_n)_{n\in \bb N}\in \widehat \Pi$ has been cleverly chosen to construct the example. The significance of choosing this sequence lies in the fact that it guarantees the existence of a natural number $M$ so that $s_M(K)>s_{M+1}(K)$ as well as $\frac{\pi_M}{\pi_{M+1}}>\frac{s_M(P)}{s_{M+1}(P)}$. We use this example as a tool to prove the following proposition.
\end{remark}

\begin{prop}
Let $P\in \cl B(\cl H)$ be a positive operator. If $\pi\in\widehat \Pi$ such that $$\frac{\pi_n}{\pi_{n+1}}>\frac{s_n(P)}{s_{n+1}(P)}\,\text{ for every }\, n\in \bb N,$$
then $P\notin \cl N_{\Phi_\pi^*}(\cl H)$.
\end{prop}

\begin{proof}
 To show that $P\notin \cl N_{\Phi_\pi^*}(\cl H)$, we assume that $P\in \cl N_{\Phi_\pi^*}(\cl H)$,
and we deduce a contradiction from this assumption. 
If $P\in \cl N_{\Phi_\pi^*}(\cl H)$, then there exists $K=\text{diag}(s_1(K),s_2(K),...)$ in $\cl B_1(\cl H)$ with $\|K\|_{\Phi_\pi}=\sum_j\pi_js_j(K)=1$ such that $\|P\|_{\Phi_\pi^*}=|\tr(PK)|=\sum_js_j(P)s_j(K)$. 
Since $K\in \cl B_1(\cl H)\subseteq \cl B_0(\cl H)$, we have $\lim_{j\rightarrow \infty} s_j(K)=0$. This forces the existence of a natural number $M$ such that $s_M(K)>s_{M+1}(K).$ 
We complete the proof by establishing the existence of an operator $\tilde K\in \cl B_1(\cl H), \,\|\tilde K\|_{\Phi_\pi}=1$ of the form $\tilde K=\text{diag}\{s_1(\tilde K),s_2(\tilde K),...\}$ such that $\sum_js_j(P)s_j(\tilde K)> \sum_js_j(P)s_j(K)$. To this end, we define 
$$
\lambda:=\frac{\sum_{j=M}^{M+1}\pi_js_j(K)}{\sum_{j=M}^{M+1}\pi_j},
$$ 
and let 
$$
\tilde K:=\begin{pmatrix}
s_1(K)\\
&\ddots&&\\
&&s_{M-1}(K)&&\\
&&&\lambda&\\
&&&&\lambda\\
&&&&&s_{M+1}(K)\\
&&&&&&\ddots
\end{pmatrix}.
$$
It can be verified that $\|\tilde K\|_{\Phi_\pi}=\|K\|_{\Phi_\pi}=1$ so that $\tilde K\in \cl B_1(\ell^2)$ and is of the form $\tilde K=\text{diag}\{s_j(\tilde K)\}$. However, since
$$
\frac{\pi_n}{\pi_{n+1}}>\frac{s_n(P)}{s_{n+1}(P)}\,\text{ for every }\, n\in \bb N,
$$
it follows that 
$$
\frac{\pi_M}{\pi_{M+1}}>\frac{s_M(P)}{s_{M+1}(P)},
$$ 
and thus we have,
$$
s_M(P)s_M(\tilde K)+s_{M+1}(P)s_{M+1}(\tilde K)>s_M(P)s_M(K)+s_{M+1}(P)s_{M+1}(K),
$$
which yields
 $$\sum_{j}s_j(P)s_j(\tilde K)>\sum_{j}s_j(P)s_j(K)=\|P\|_{\Phi_\pi^*},$$
which contradicts the assumption that $\sum_{j}s_j(P)s_j(K)$ is the supremum of the set $$\left\{\sum_{j}s_j(P)s_j(K):K\in \cl B_1(\ell^2), K=\text{diag}\{s_1(K),s_2(K),...\}, \|K\|_{\Phi_\pi}=1\right\}.$$ This proves our assertion.

\end{proof}

\begin{thm}\label{universally-symmetric-normings-are-compacts}
Let $P\in \cl B(\cl H)$ be a positive operator and $\lim_{j\to \infty}s_j(P)\neq 0$, that is, $P$ is not compact. Then there exists $\pi\in \widehat \Pi$ such that $$
\frac{\pi_n}{\pi_{n+1}}>\frac{s_n(P)}{s_{n+1}(P)}\,\text{ for every }\, n\in \bb N.
$$

Alternatively, if $P\in \cl B(\cl H)$ is positive noncompact operator then there exists $\pi\in \widehat \Pi$ such that $P\notin \cl N_{\Phi_\pi^*}(\cl H)$.
\end{thm}

\begin{proof}
Since $P\geq 0$ and $\lim_{j\to \infty}s_j(P)\neq 0$, there exists $s>0$ such that $\lim_{j\to \infty}s_j(P)=s.$ If we take $\alpha_n:=\frac{1}{e^{1/n^2}}$ for all $n\in \bb N$ and define a sequence $\pi=(\pi_n)_{n\in \bb N}$ recursively by 
$$
\pi_1=1 \,\text{ and }\, \frac{\pi_{n+1}}{\pi_n}:=\alpha_n \frac{s_{n+1}(P)}{s_n(P)} \,\text{ for all }n\in \bb N,
$$
we have $\alpha_n<1$ for all $n\in \bb N$.
Then the fact that $s_n(P)$ is a nonincreasing sequence implies that 
$\frac{\pi_{n+1}}{\pi_n}< \frac{s_{n+1}(P)}{s_n(P)}$ for all $n\in \bb N$. Therefore, $\frac{\pi_n}{\pi_{n+1}}> \frac{s_{n}(P)}{s_{n+1}(P)}$ for every $n\in \bb N.$ All that remains is to show that $\pi\in\widehat{\Pi}$. That $\pi_1=1$ and $(\pi_n)_{n\in \bb N}$ is a strictly decreasing sequence of positive real numbers are trivial observations. We complete the proof by showing that $\lim_{n\to \infty}\pi_n> 0.$ An easy calculation shows that 
$$\pi_{n+1}=\left(\prod_{m=1}^n\alpha_m\right)\left(\frac{s_{n+1(P)}}{s_1(P)}\right)\, \text{ for each } \,n\in \bb N.$$ Let $x_n=\left(\prod_{m=1}^n\alpha_m\right)$ for every $n\in \bb N$ and observe that 
$$
\pi_{n+1}=x_n\left(\frac{s_{n+1(P)}}{s_1(P)}\right), 
$$
which yields 
$$
\lim_{n\to \infty}\pi_{n+1}=\frac{1}{s_1(P)}\lim_{n\to \infty}x_n \lim_{n\to \infty}s_{n+1}(P).
$$
This observation, together with the facts that $s_1(P)>0$ and $\lim_{n\to \infty}s_{n+1}(P)=s>0$ allows us to infer that $\lim_{n\to \infty}\pi_{n+1}>0$ if and only if $\lim_{n\to \infty}x_n>0$. But 
$$\lim_{n\to \infty}x_n =
\lim_{n\to \infty}\frac{1}{e^{\sum_{m=1}^n1/m^2}}
=\frac{1}{e^{\pi/6}}>0,
$$ and we conclude that $\lim_{n\to \infty}\pi_{n}>0$. This completes the proof.

\end{proof}

We are now in a position to prove a key result --- a characterization theorem for positive operators in $\{\cl N_{\Phi_\pi^*}(\cl H):\pi\in \widehat{\Pi}\}$ --- which answers the question we asked in the paragraph preceding the Example \ref{Intuitive-to-minicharacterization-for-positives}. Moreover, this result is a special case of a more general result that is presented in the next section (see Theorem \ref{complete-characterization-positives}).

\begin{thm}\label{mini-characterization-positives}
Let $P$ be a positive operator on $\cl{H}$. Then the following statements are equivalent.
\begin{enumerate}
\item $P\in \cl B_0(\cl H)$.
\item $P\in\cl{AN}_{\Phi_\pi^*}(\cl H)$ for every $\pi\in \widehat \Pi$.
\item $P\in\cl{N}_{\Phi_\pi^*}(\cl H)$ for every $\pi\in \widehat \Pi$.
\end{enumerate}
\end{thm}

\begin{proof}
(1) implies (2) follows from Theorem \ref{compacts-are-absolutely-symmetric-norming}. (2) implies (3) is obvious. (3) implies (1) is a direct consequence of the Theorem \ref{universally-symmetric-normings-are-compacts}.
\end{proof}

We conclude this section by proving the following result that extends the above theorem to bounded operators in $\cl B(\cl H)$, the above theorem required the operator to be positive. This is the main theorem of this section.
\begin{thm}\label{mini-characterization-arbitrary}
If $T\in \cl B(\cl H)$, then the following statements are equivalent.
\begin{enumerate}
\item $T\in \cl B_0(\cl H)$.
\item $T\in\cl{AN}_{\Phi_\pi^*}(\cl H)$ for every $\pi\in \widehat \Pi$.
\item $T\in\cl{N}_{\Phi_\pi^*}(\cl H)$ for every $\pi\in \widehat \Pi$.
\end{enumerate}
\end{thm}

\begin{proof}
(2) implies (3) is obvious, as is (1) implies (2) from the Theorem \ref{compacts-are-absolutely-symmetric-norming}. The Proposition \ref{T-iff-|T|-phi-star} along with the Theorem \ref{universally-symmetric-normings-are-compacts} proves (3) implies (1).
\end{proof}

The above result, although very important, is transitory. We will see a much more general result than this --- the characterization theorem for universally symmetric norming operators (see Theorem \ref{complete-characterization-arbitrary}).
\section{Universally symmetric norming operators and their characterization}\label{Characterization-Universally-Symmetrically-Norming}

In the preceding section we considered a certain family $\{\Phi_\pi:\pi\in \widehat{\Pi}\}$ of s.n.functions and a family of symmetric norms on $\cl B(\cl H)$ generated by the dual of these, and we studied the symmetric norming operators and absolutely symmetric norming operators with respect to each of these symmetric norms. The fact that each member of the family $\{\Phi_\pi:\pi\in \widehat{\Pi}\}$ is equivalent to the maximal s.n.function $\Phi_1$ suggests the possibility of extending the Theorem \ref{mini-characterization-arbitrary} to a larger family of s.n.functions. With this in mind, our attention is drawn to the family of all s.n.functions that are equivalent to the maximal s.n.function, that is, the family $\{\Phi:\Phi\text{ is equivalent to } \Phi_1\}$  of s.n.functions. This larger family of s.n.functions provides us with the leading idea on which we develop the notions of ``universally symmetric norming operators'' and ``universally absolutely symmetric norming operators'' on a separable Hilbert space. The study of these operators are taken up in this section. Our main result is Theorem \ref{complete-characterization-arbitrary} which, in effect, states that an operator in $\cl B(\cl H)$ is universally symmetric norming if and only if it is universally absolutely symmetric norming, which holds if and only if it is compact.

We begin by defining the relevant classes of operators.

\begin{defn}\label{Universally-symmetric-Norming}
An operator $T\in \cl B(\cl H)$ is said to be \emph{universally symmetric norming} if
$T\in \cl N_{\Phi^*}(\cl H)$ for every s.n.function $\Phi$ equivalent to the maximal s.n.function $\Phi_1$.
Alternatively, an operator $T\in (\cl B(\cl H))$ is said to be universally symmetric norming if 
$T\in \cl N_{\Phi^*}(\cl H)$ for every $\Phi$ from the family $\{\Phi:\Phi\text{ is equivalent to } \Phi_1\}$ of s.n.functions.
\end{defn}

\begin{defn}\label{Universally-absolutely-symmetric-Norming}
An operator $T\in \cl B(\cl H)$ is said to be \emph{universally absolutely symmetric norming} if
$T\in \cl{AN}_{\Phi^*}(\cl H)$ for every s.n.function $\Phi$ equivalent to the maximal s.n.function $\Phi_1$.
\end{defn}

\begin{remark}
Since every symmetric norm on $\cl B(\cl H)$  is topologically equivalent to the usual operator norm, it follows that $T\in \cl B(\cl H)$ is universally symmetric norming (respectively universally absolutely symmetric norming) if and only if $T$ is symmetric norming (respectively absolutely symmetric norming) with respect to every symmetric norm on $\cl B(\cl H)$. Another important observation worth mentioning here is that every universally absolutely symmetric norming operator is universally symmetric norming.
\end{remark}
The following theorem gives a useful characterization of positive universally symmetric norming operators in $\cl B(\cl H)$.

\begin{thm}\label{complete-characterization-positives}
Let $P$ be a positive operator on $\cl{H}$ and let $\Phi_1$ denote the maximal s.n.function. Then the following statements are equivalent.
\begin{enumerate}
\item $P\in \cl B_0(\cl H)$.
\item $P$ is universally absolutely symmetric norming, that is,
 $P\in\cl{AN}_{\Phi^*}(\cl H)$ for every s.n.function $\Phi$ equivalent to $\Phi_1$.
\item $P$ is universally symmetric norming, that is, $P\in\cl{N}_{\Phi^*}(\cl H)$ for every s.n.function $\Phi$ equivalent to $\Phi_1$.
\end{enumerate}
\end{thm}

\begin{proof}
The implication $(1)\implies (2)$ is an immediate consuequence of Theorem \ref{compacts-are-absolutely-symmetric-norming} and $(2)\implies (3)$ is straightforward. To prove $(3)\implies (1)$, assume that the positive operator $P$ is universally symmetric norming on $\cl H$. Then  the statement $(3)$ of Theorem \ref{mini-characterization-positives} holds which implies that $P$ is compact and the proof is complete.
\end{proof}

We next establish the following result which allows us to extend the above theorem to operators that are not necessarily positive.

\begin{prop}\label{T-iff-|T|-universal}
An operator $T\in \cl B(\cl H)$ is universally symmetric norming if and only if $|T|$ is universally symmetric norming.
\end{prop}

\begin{proof}
This follows immediately from the Proposition \ref{T-iff-|T|-phi-star}.
\end{proof}

We are now prepared to extend the Theorem \ref{complete-characterization-positives} for an arbitrary operator on a separable Hilbert space. 

\begin{thm}\label{complete-characterization-arbitrary}
Let $T\in \cl B(\cl{H})$ and let $\Phi_1$ denote the maximal s.n.function. Then the following statements are equivalent.
\begin{enumerate}
\item $T\in \cl B_0(\cl H)$.
\item $T$ is universally absolutely symmetric norming, that is,
 $T\in\cl{AN}_{\Phi^*}(\cl H)$ for every s.n.function $\Phi$ equivalent to $\Phi_1$.
\item $T$ is universally symmetric norming, that is, $T\in\cl{N}_{\Phi^*}(\cl H)$ for every s.n.function $\Phi$ equivalent to $\Phi_1$.
\end{enumerate}
\end{thm}

\begin{proof}
Theorem \ref{complete-characterization-positives} and the preceding proposition yield this result.
\end{proof}

\begin{remark} The preceding result provides an alternative characterization of compact operators on $\cl H$.
\end{remark}

It is worth noticing that Theorem \ref{universally-symmetric-normings-are-compacts} essentially states that given any positive noncompact operator on (an infinite-dimensional separable) Hilbert space $\cl H$, there exists a symmetric norm on $\bh$ with respect to which the operator does not attain its norm. The following corollary extends Theorem \ref{universally-symmetric-normings-are-compacts} to any noncompact operator. 

\begin{cor}
If $T\in \cl B(\cl H)$ is a noncompact operator then there exists $\pi\in \widehat \Pi$ such that $T\notin \cl N_{\Phi_\pi^*}(\cl H)$
\end{cor}

\begin{proof}
Contrapositively, if for every $\pi\in \widehat \Pi$ the operator $T\in \cl N_{\Phi_\pi^*}(\cl H)$, then by the preceding theorem $T$ must be a compact operator.
\end{proof}

\end{document}